\documentclass[a4paper]{article}
\usepackage{mathptmx}
\usepackage{amscd,amssymb,amsmath,amsthm}
\usepackage{textcomp}

\makeatletter
\renewcommand{\@seccntformat}[1]
{{\csname the#1\endcsname}.\hspace{0.3em}}

\renewcommand{\section}{\@startsection
{section}
{1}
{0mm}
{-1.5\baselineskip}
{\baselineskip}
{\bfseries\normalsize}}

\renewcommand{\subsection}{\@startsection
{subsection}
{2}
{0mm}
{-\baselineskip}
{0.5\baselineskip}
{\normalsize\itshape}}

\makeatother

\theoremstyle{plain}
\newtheorem{theorem}{Theorem}[section]
\newtheorem{corollary}[theorem]{Corollary}
\newtheorem{lemma}[theorem]{Lemma}
\newtheorem{prop}[theorem]{Proposition}
\newtheorem{claim}{Claim}[section]

\newtheorem*{BT}{Bubble convergence theorem}

\theoremstyle{definition}
\newtheorem*{defin*}{Definition}

\theoremstyle{remark}
\newtheorem{example}{Example}[section]
\newtheorem*{example*}{Example}
\newtheorem*{remark*}{Remark}

\DeclareMathAlphabet{\matheur}{U}{eur}{m}{n}
\DeclareMathAlphabet{\matheus}{U}{eus}{m}{n}
\DeclareMathAlphabet{\matheuf}{U}{euf}{m}{n}

\numberwithin{equation}{section}

\newcommand{\abs}[1]{\left\lvert#1\right\rvert}

\DeclareMathOperator{\Hom}{Hom}

\DeclareMathOperator{\diam}{diam}

\begin{document}

\author{Gerasim  Kokarev
\\ {\small\it D\'epartement de Math\'ematiques, Universit\'e de
  Cergy-Pontoise}
\\ {\small\it Site de Saint Martin, 2, avenue Adolphe Chauvin}
\\ {\small\it 95302 Cergy-Pontoise Cedex, France}
\\ {\small\it Email: {\tt Gerasim.Kokarev@u-cergy.fr}}
}

\title{Curvature and bubble convergence of harmonic maps}
\date{}
\maketitle

\begin{abstract}
\noindent
We explore geometric aspects of bubble convergence for harmonic maps. More precisely, we show that the formation of bubbles is characterised by the local excess of curvature on the target manifold. We give a universal estimate for curvature concentration masses at each bubble point  and show that there is no curvature loss in the necks.  Our principal hypothesis is that the target manifold is K\"ahler.
\end{abstract}


\section{Introduction}
\label{intro}
Let $\Sigma$ and $M$ be a closed Riemannian surface and a closed Riemannian manifold respectively. For a map $u:\Sigma\to M$ we consider the energy functional
$$
E(u)=\frac{1}{2}\int_\Sigma\abs{du}^2d\mathit{Vol}_\Sigma,
$$
where $\abs{du}^2$ denotes the norm of the differential $du$ determined by the metrics on $\Sigma$ and $M$. The critical points of $E(u)$ are called {\em harmonic maps}. For a sequence of harmonic maps of bounded energy its weak limit in $W^{1,2}(\Sigma,M)$ is described by the following celebrated theorem of Sacks and Uhlenbeck~\cite{SaU}, see also~\cite{Jo,Pa96}.
\begin{BT}
Let $u_n:\Sigma\to M$, where $n=1, 2, \ldots$, be a sequence of harmonic maps of bounded energy. Then it contains a subsequence (also denoted by $u_n$) that converges weakly in $W^{1,2}(\Sigma,M)$ to a harmonic map $u:\Sigma\to M$. Moreover, there exists a finite number of points $\{x_1, \ldots, x_k\}$ in $\Sigma$ such that $u_n$ converges to $u$ in $C^1$-topology uniformly on compact subsets in the complement of $\{x_1, \ldots, x_k\}$ and
$$
\frac{1}{2}\abs{du_n}^2d\mathit{Vol}_\Sigma\rightharpoonup\frac{1}{2}\abs{du}^2d\mathit{Vol}_\Sigma
+\sum_{i=1}^{k}m_i\delta_{x_i}\quad\text{as measures},
$$
where each $m_i>0$ is a sum of the energies of non-constant harmonic $2$-spheres in $M$.
\end{BT}

The theorem says that the loss of energy in the limit is concentrated precisely at the points $\{x_1, \ldots, x_k\}$, called {\em bubble points}. Moreover, as is known~\cite{SaU}, there is a constant $\varepsilon_*>0$,  depending on the geometry of $M$, such that $m_i\geqslant \varepsilon_*$.  The fact that the $m_i$'s are sums of the energies of harmonic $2$-spheres, so-called {\em bubbles},  allows to obtain more explicit estimates for these quantities in some examples. For instance, if $M$ is a complex projective space, more generally, a complex Grassmannian, or a unitary group, the values $m_i$ are multiples of $4\pi$. The important part of the theorem is the so-called {\em energy identity}: there is no loss of energy in the necks between the limit map $u$ and the bubbles or between the bubbles themselves. For other results describing the formation of bubbles and  applications we refer to~\cite{Pa96, CT}.
Similar bubbling phenomena have been discovered and studied for a number of other equations, see~\cite{Gro,Lions85}, and have been central issues in the non-linear analysis for decades since the work~\cite{SaU}.

The purpose of this paper is to study geometric aspects of this phenomenon for harmonic maps. To illustrate 
the basic link between curvature and bubbling we recall the following fact. As is known~\cite{Sch}, if a domain $D\subset\Sigma$ is mapped by the $u_n$'s to a region in $M$ where the sectional curvature is non-positive, then the norms $\abs{du_n}(x)$, where $x\in D$, can be bounded in terms of the energy $E(u_n)$. In particular, under the hypotheses of the theorem  no bubbles can appear in $D$. Thus, intuitively one expects that the {\em appearance of bubbles is related to the local excess of positive curvature} on the images of maps. Below we make this intuition precise when the target manifold $M$ is K\"ahler. In addition, we obtain explicit lower bounds for the energies of harmonic maps, and in particular, lower bounds for the energy concentration masses $m_i$.

Now we outline the organisation of the paper. The principal results are stated in the following Sect.~\ref{res} below. First, we introduce the {\em curvature densities} of harmonic maps and describe their properties. These quantities are used to detect bubbling in the sequences of harmonic maps, and besides, yield explicit energy bounds. Then we state the main result, Theorem~\ref{MT}. Loosely speaking, it says that the energy and curvature densities blow up at the same collection of points. It also gives a universal estimate for the amount of curvature consumed by each bubble. In Sect.~\ref{prem} we give a background material on harmonic maps valued in K\"ahler manifolds, and explain the properties of the curvature densities. The proof of the main result, Theorem~\ref{MT}, appears in the following Sect.~\ref{proofMT0} and~\ref{proofMT}. The latter also contains an account on the bubble tree construction for sequences of harmonic maps, necessary to explain the ``no curvature loss in the necks'' statement. 

This paper is a part of author's programme to link the blow-up analysis  for the harmonic map problem to the geometry of underlying manifolds. In a forthcoming paper we shall study analogous links between curvature and bubbling in the harmonic map heat flow.

\smallskip
\noindent{\em Acknowledgements.} During the work on the paper the author has been supported by the EU Commission through the Marie Curie actions scheme.

\section{Statements of the main results}
\label{res}
\subsection{K\"ahler targets and curvature densities}
Let $\Sigma$ and $M$ be a Riemannian surface and a K\"ahler manifold endowed with the K\"ahler forms $\omega_\Sigma$ and $\omega_M$ respectively. For a smooth map $u:\Sigma\to M$ denote by
$$
\partial u:T^{1,0}\Sigma\to T^{1,0}M\quad\text{and}\quad \bar\partial u:T^{0,1}\Sigma\to T^{1,0}M
$$
the corresponding components of the complexified differential $du$. By $e(u)$ we also denote the energy density $(1/2)\abs{du}^2$, and by $e'(u)$ and $e''(u)$ the squared Hermitian norms $\abs{\partial u}^2$ and $\abs{\bar\partial u}^2$. They satisfy the relations
\begin{equation}
\label{erels}
e(u)=e'(u)+e''(u)\quad\text{and}\quad\left(e'(u)-e''(u)\right)\omega_\Sigma=u^*\omega_M,
\end{equation}
see Sect.~\ref{prem} for details.

Let $u:\Sigma\to M$ be a harmonic map, that is a solution of the Euler-Lagrange equation for the energy
$$
E(u)=\int_Me(u)d\mathit{Vol}_\Sigma.
$$
As is known~\cite{To,Wood}, each function $e'(u)$ and $e''(u)$ is either identically zero (and then $u$ is anti-holomorphic or holomorphic respectively) or has only isolated zeroes. Moreover, the isolated zeroes have a well-defined order $r$ -- in local coordinates near zeroes the functions $e'(u)$ and $e''(u)$ have the form $z^{2r}\abs{\psi}^2(z)$, where $\abs{\psi}^2(z)$ does not vanish.
\begin{defin*}
Let $u:\Sigma\to M$ be a harmonic map, and suppose that each function $e'(u)$ and $e''(u)$ has isolated zeroes only. By the {\em curvature densities} $q'(u)$ and $q''(u)$ we call the functions defined by the following relations:
$$
-q'(u)e'(u)\omega_\Sigma=\langle u^*\Omega(\partial u(v)),\partial u(v)\rangle,
$$
$$
~~q''(u)e''(u)\omega_\Sigma=\langle u^*\Omega(\bar\partial u(\bar v)),\bar\partial u(\bar v)\rangle.
$$
Here $\Omega$ is the curvature form on $M$, $\langle\cdot,\cdot\rangle$ is the scalar product induced by the Hermitian metric on $M$, and $v$ is a unit vector in $T^{1,0}\Sigma$. If either $e'(u)$ or $e''(u)$ is identically zero, then we set $q'(u)$ or $q''(u)$  to be zero identically respectively.
\end{defin*}

The relations above define the curvature densities $q'(u)$ and $q''(u)$, as {\em certain norms} of the pull-back $u^*\Omega$ of the curvature form, everywhere except for the zeroes of $e'(u)$ and $e''(u)$ respectively. As is shown in Sect.~\ref{prem}, these singularities are removable. The amount of positive curvature related to the energy concentration in the bubble convergence theorem will be measured by the non-negative parts of $q'(u)$ and $q''(u)$; they are denoted by $q'_+(u)$ and $q''_+(u)$ respectively,
$$
q'_+(u)=\max\{q'(u),0\}\quad\text{and}\quad q''_+(u)=\max\{q''(u),0\}.
$$
By $q_+(u)$ we denote the total positive curvature density, that is the sum $q'_+(u)+q''_+(u)$.

Before proceeding we make two remarks. First, intuitively the curvature densities behave like functions that are quadratic with respect to the first derivatives of $u$. In particular, the Cauchy-Schwarz inequality yields the estimates:
\begin{equation}
\label{CS1}
q'_+(u)\leqslant\abs{q'(u)}\leqslant\sqrt{2}\!\abs{\Omega}e(u),
\end{equation}
\begin{equation}
\label{CS2}
q''_+(u)\leqslant\abs{q''(u)}\leqslant\sqrt{2}\!\abs{\Omega}e(u),
\end{equation}
where $\abs{\Omega}$ stands for the norm of the curvature operator $\Lambda^{1,1}T^*M\to \Lambda^{1,1}T^*M$ in the natural Hermitian metric on the space of $(1,1)$-forms, see Sect.~\ref{prem}. Second, as follows from definitions, the integrals
$$
Q'_+(u)=\int_\Sigma q'_+(u)d\mathit{Vol}_\Sigma\quad\text{and}\quad 
Q''_+(u)=\int_\Sigma q''_+(u)d\mathit{Vol}_\Sigma
$$
are invariant under the conformal change of a metric on the domain $\Sigma$, and so is their sum
$Q_+(u)$.

The following result gives bounds for $Q_+(u)$ via the genus $\varrho$ of $\Sigma$, and the multiplicities of the zeroes of $\partial u$ and $\bar\partial u$.
\begin{theorem}
\label{qbounds}
Let $\Sigma$ and $M$ be a closed Riemannian surface and a K\"ahler manifold respectively. Suppose that a harmonic map $u:\Sigma\to M$ is not anti-holomorphic. Then we have
$$
Q'_+(u)\geqslant\pi\left(\sum_{\partial u(z_i)=0}r'_i+2-2\varrho\right),
$$
where $r'_i$ is the multiplicity of $z_i$, and $\varrho$ is the genus of $\Sigma$. If a map $u$ is not holomorphic, then we have an analogous inequality for $Q''_+(u)$:
$$
Q''_+(u)\geqslant\pi\left(\sum_{\bar\partial u(z_i)=0}r''_i+2-2\varrho\right).
$$
\end{theorem}
\begin{corollary}
\label{cor:qbounds}
Let $u$ be a harmonic $2$-sphere in a K\"ahler manifold $M$. If $u$ is not anti-holomorphic (or holomorphic), then the quantity $Q'_+(u)$ (or $Q''_+(u)$) is at least $2\pi$.
\end{corollary}

The combination of these inequalities with the estimates~\eqref{CS1}--\eqref{CS2} gives lower bounds for the energies of harmonic maps, which seem to have been unnoticed. In particular, for a non-trivial harmonic $2$-sphere $u$, we have
\begin{equation}
\label{sphereMB}
E(u)\geqslant\sqrt{2}\pi\left(\max\abs{\Omega}\right)^{-1},
\end{equation}
where $\abs{\Omega}$ stands for the norm of the curvature operator whose maximum is taken over all points in $M$. This implies also the lower bound for the energy concentration masses $m_i$ in the Sacks-Uhlenbeck theorem, and prevents the occurrence of bubbling if the limit amount of energy is strictly less than the right-hand side in~\eqref{sphereMB}.

\subsection{Bubble convergence via curvature densities}
Let $u_n:\Sigma\to M$ be a sequence of harmonic maps of bounded energy. After a selection of a subsequence we can suppose that it converges to a harmonic map $u$ in the sense of the bubble convergence. Further, by~\eqref{CS1}--\eqref{CS2} the integrals  $Q_+(u_n)$ are bounded, and we can suppose that the measures $q_+(u_n)d\mathit{Vol}_\Sigma$ converge weakly to some measure $dQ_+$. Our main result below shows that the singular part of $dQ_+$ is supported precisely at the set of  bubble points and establishes an identity for $dQ_+$.
\begin{theorem}
\label{MT}
Let $\Sigma$ and $M$ be a closed Riemannian surface and a closed K\"ahler manifold respectively. Let $u_n:\Sigma\to M$ be a sequence of harmonic maps of bounded energy that converges to a harmonic map $u$ in the sense of bubble convergence, and let $\{x_1, \ldots, x_k\}$ be the set of the corresponding bubble points. Then the limit measure $dQ_+$ satisfies the identity
$$
dQ_+=q_+(u)d\mathit{Vol}_\Sigma+\sum_{i=1}^k q_i\delta_{x_i},
$$
where each $q_i\geqslant2\pi$. Moreover, each quantity $q_i$ is equal to the sum of $Q_+(\phi)\geqslant 2\pi$ over all non-trivial bubble spheres $\phi$ corresponding to the point $x_i$.
\end{theorem} 

The theorem says that at a given point $x\in\Sigma$ the energy blows up if and only if there is a "sufficient amount" of positive curvature near the images $u_n(x)$. Moreover, each non-trivial bubble sphere $\phi$ consumes a non-trivial (universally estimated!) amount $Q_+(\phi)$ of positive curvature off the sequence $u_n$. To illustrate the result consider the case when all harmonic maps $u_n:\Sigma\to M$ in the bubble convergence theorem are holomorphic. Then the curvature density $q_+(u_n)$ is given by the explicit formula
$$
q_+(u_n)=\frac{1}{2}H_+(\partial u_n(v),\overline{\partial u_n(v)})\cdot e(u_n),
$$
see Sect.~\ref{prem}, where $H_+(\cdot,\cdot)$ stands for the non-negative part of the holomorphic sectional curvature, and $v$ is a unit vector in $T^{1,0}\Sigma$. If the holomorphic sectional curvature changes sign or has vanishing directions, then the term $H_+(\cdot,\cdot)$ above can, in principle, vanish or approach zero as $n\to+\infty$. However, by Theorem~\ref{MT} it does not affect the blow up of the right-hand side  -- the curvature density blows up if and only if the energy density $e(u_n)$ does. Further, if $H$ is an upper bound for the holomorphic sectional curvature of $M$, then for a holomorphic map $u$ we have the following version of inequality~\eqref{CS1}:
$$
q_+(u)\leqslant \frac{H}{2}\cdot e(u).
$$
In particular, if $H\geqslant 0$, then the energy of a holomorphic (or anti-holomorphic) sphere $u$ satisfies the bound
$$
E(u)\geqslant 4\pi/H.
$$
The right-hand side here is also the lower bound for the energy concentration masses $m_i$ at a bubble point $x_i$  whose bubble tree decomposition contains a bubble which is either holomorphic or anti-holomorphic.

We end the discussion with a couple of examples.
\begin{example}
Let $M$ be a $2$-dimensional sphere endowed with an arbitrary Riemannian metric. Denote by $K_M$ its Gauss curvature, and by $K_M^+$ the non-negative part $\max\{ K_M,0\}$. In this case the curvature densities $q'(u)$ and $q''(u)$ are given by explicit formulas, see Sect.~\ref{prem}. Their analysis shows that the energy estimate~\eqref{sphereMB} for a bubble sphere can be improved to
$$
E(u)\geqslant 4\pi\left(\max K_M^+\right)^{-1}.
$$
Besides, the blow-ups of the curvature densities $q'_+(u)$ and $q''_+(u)$ are responsible precisely for the occurrence of holomorphic and anti-holomorphic bubbles respectively.
\end{example}
\begin{example}
Let $M$ be a complex projective space $\mathbf{CP}^n$ with constant holomorphic sectional curvature $c$.
The formulas for the curvature densities in Sect.~\ref{prem} show that for any (not necessarily holomorphic or anti-holomorphic) harmonic non-trivial sphere its energy is bounded below by $4\pi/c$. As in the example above, the blow-up of the quantity $q'_+(u)$ (or $q''_+(u)$) implies the blow-up of the corresponding part of the energy $e'(u)$ (or $e''(u)$ respectively). Besides, if for a given bubble point $x_i$ every bubble $\phi$ in its bubble tree decomposition is such that
$$
d\phi(T_xS^2)\text{ is contained in a }1\text{-dimensional complex subspace of }T_{u(x)}\mathbf{CP}^n
$$
for every $x\in S^2$, then the converse holds; that is, the quanity $q'_+(u)$ blows up if and only if so does $e'(u)$.
\end{example}

\section{Harmonic maps into K\"ahler manifolds}
\label{prem}
\subsection{Preliminaries}
Let $(\Sigma, g)$ and $(M, h)$ be a Riemann surface and a K\"ahler manifold respectively.
In local complex coordinates the Hermitian metrics $g$ and $h$ have the forms
$$
g=g\cdot dzd\bar z\quad\text{and}\quad h=h_{\alpha\bar\beta}du^\alpha du^{\bar\beta}.
$$
For a smooth map $u:\Sigma\to M$ the differentials $\partial u$ and $\bar\partial u$ in local coordinates correspond to the matrices $(u^\alpha_z)$ and $(u^\alpha_{\bar z})$. In particular, the holomorphic and anti-holomorphic parts of the energy $e'(u)$ and $e''(u)$ in this notation are given by the formulas
\begin{equation}
\label{energies}
e'(u)=\frac{1}{g}h_{\alpha\bar\beta}u^\alpha_zu^{\bar\beta}_{\bar z}\quad\text{and}
\quad e''(u)=\frac{1}{g}h_{\alpha\bar\beta}u^\alpha_{\bar z}u^{\bar\beta}_{z}.
\end{equation}
The relations in~\eqref{erels} follow now by straightforward calculation; the details can be found in~\cite{EL83}. 

The {\em harmonic map equation}, that is the Euler-Lagrange equation for the energy $E(u)$, has the form
\begin{equation}
\label{hme}
u^\beta_{z\bar z}+\Theta^\beta_{\alpha\gamma}u^\alpha_zu^\gamma_{\bar z}=0,
\end{equation}
where $(\Theta^\beta_{\alpha\gamma})$ stand for the connection form $\Theta=\Theta^\beta_{\alpha\gamma}du^\gamma$ corresponding to the Hermitian connection on $T^{1,0}M$, see~\cite{EL83}. If $(\Omega^\alpha_\beta)$ is the curvature form matrix
$$
\Omega^\alpha_\beta=K^\beta_{\alpha\gamma\bar\delta}du^\gamma\wedge du^{\bar\delta},
\quad\text{where }K^\beta_{\alpha\gamma\bar\delta}=\frac{\partial~}{\partial z^{\bar\delta}}
(\Theta^\beta_{\alpha\gamma}),
$$
then the curvature densities, introduced in Sect.~\ref{res}, can be written as
\begin{equation}
\label{q'}
-q'(u)e'(u)=\frac{1}{g^2}K_{\alpha\bar\beta\gamma\bar\delta}(u^\gamma_zu^{\bar\delta}_{\bar z}-u^\gamma_{\bar z}u^{\bar\delta}_z)u^\alpha_zu^{\bar\beta}_{\bar z},
\end{equation}
\begin{equation}
\label{q''}
q''(u)e''(u)=\frac{1}{g^2}K_{\alpha\bar\beta\gamma\bar\delta}(u^\gamma_zu^{\bar\delta}_{\bar z}-u^\gamma_{\bar z}u^{\bar\delta}_z)u^\alpha_{\bar z}u^{\bar\beta}_{z}.
\end{equation}
As was observed by Wood~\cite{Wood}, for a harmonic map $u:\Sigma\to M$ the singularities of $q'(u)$ and $q''(u)$ at the zeroes of $e'(u)$ and $e''(u)$ respectively are removable. This follows from the fact that the vector bundle $u^*T^{1,0}M$ admits a holomorphic structure such that $\partial u$ is its holomorphic section, see~\cite{KM}. Hence, near an isolated zero it can be written in the form $z^r\psi(z)$, where $\psi(z)$ is a non-vanishing section, and the claim follows from relation~\eqref{q'}. The statement for $q''(u)$ is obtained in a similar fashion.

We proceed with the examples when the curvature expressions on the right-hand sides in~\eqref{q'} -\eqref{q''} can be easily analysed. First, for a holomorphic map $u$ we have the following explicit formula 
\begin{equation}
\label{holq}
q'(u)=\frac{1}{2}H(\partial u(v),\overline{\partial u(v)})e'(u),
\end{equation}
where $H(\cdot,\cdot)$ denotes the holomorphic sectional curvature,
 $$
\frac{1}{2}H(X,\bar X)=-\langle K(X, \bar X) X, \bar X\rangle/\abs{X}^4,
$$
and $v$ is a unit vector in $T^{1,0}\Sigma$. An analogous relation for $q''(u)$ holds also for anti-holomorphic maps. When $M$ has complex dimension one, the curvature components $K_{\alpha\bar\beta\gamma\bar\delta}$ reduce to the one non-trivial value $(-1/2)K_M$, where $K_M$ is the Gauss curvature. In this case the quantities $q'(u)$ and $q''(u)$ are given by the following formulas:
$$
q'(u)=\frac{1}{2}K_M(e'(u)-e''(u)),
$$
$$
q''(u)=-\frac{1}{2}K_M(e'(u)-e''(u)).
$$
Further, if the holomorphic sectional curvature of $M$ is constant and is equal to $c$, then
$$
K_{\alpha\bar\beta\gamma\bar\delta}=-\frac{c}{4}(h_{\alpha\bar\beta}h_{\gamma\bar\delta}+
h_{\alpha\bar\delta}h_{\gamma\bar\beta}),
$$
and $q'(u)$ and $q''(u)$ have the form
$$
q'(u)=\frac{c}{2}\left((e'(u)-e''(u))+\sigma(u)e''(u)\right),
$$
$$
q''(u)=-\frac{c}{2}\left((e'(u)-e''(u))-\sigma(u)e'(u)\right),
$$
where $0\leqslant\sigma(u)\leqslant 1/2$. Moreover, as is shown in~\cite{To,Wood}, the value
$\sigma(u)(x)$ vanishes if and only if the image  of $du(x)$ is contained in a complex one-dimensional subspace of $T_{u(x)}M$. 

We proceed with the Bochner-type formulas for $\Delta e'(u)$ and $\Delta e''(u)$ due to~\cite{To, Wood}. First, we introduce more notation. By the Laplacian $\Delta$ we mean below the non-negative operator on $\Sigma$, which in local coordinates has the form
$$
\Delta=-\frac{4}{g}\frac{\partial^2~}{\partial z\partial\bar z}.
$$
For a map $u:\Sigma\to M$ by $\beta'(u)$ and $\beta''(u)$ we denote the $(1,0)$- and $(0,1)$-parts of the derivatives of $\partial u$ and $\bar\partial u$ in the natural connections on the bundles
$\Hom(T^{1,0}\Sigma, T^{1,0}M)$ and $\Hom(T^{0,1}\Sigma, T^{1,0}M)$ respectively. The proofs of the following statements can be found in~\cite{To,Wood}.
\begin{prop}
\label{to}
Let $\Sigma$ and $M$ be a Riemannian surface and a K\"ahler manifold respectively. Then 
for any harmonic map $u:\Sigma\to M$ we have the following relations
$$
\frac{1}{4}\Delta e'(u)=-\abs{\beta'(u)}^2+q'(u)e'(u)-\frac{1}{2}K_\Sigma e'(u),
$$
$$
\frac{1}{4}\Delta e''(u)=-\abs{\beta''(u)}^2+q''(u)e'(u)-\frac{1}{2}K_\Sigma e''(u),
$$
where $K_\Sigma$ is the Gauss curvature of $\Sigma$.
\end{prop}
\begin{corollary}
\label{log_to}
For a harmonic map $u:\Sigma\to M$ the following relations hold
$$
\frac{1}{4}\Delta\log e'(u)=-\alpha'(u)+q'(u)-\frac{1}{2}K_\Sigma,
$$
$$
\frac{1}{4}\Delta\log e''(u)=-\alpha''(u)+q''(u)-\frac{1}{2}K_\Sigma,
$$
where $\alpha'(u)$ and $\alpha''(u)$ are non-negative quantities. The relations are understood to hold at every point in $\Sigma$ where $e'(u)\ne 0$ or $e''(u)\ne 0$ respectively. 
\end{corollary}

Finally, note that if a map $u$ is holomorphic, then by~\eqref{holq} the first relation in Prop.~\ref{to} takes the form
\begin{equation}
\label{chern}
\frac{1}{4}\Delta e'(u)=-\abs{\beta'(u)}^2+\frac{1}{2}H(\partial u(v),\overline{\partial u(v)})(e'(u))^2-\frac{1}{2}K_\Sigma e'(u).
\end{equation}
This is a known Bochner-type identity for holomorphic maps due to~\cite{Ch,Lu} for which the K\"ahler hypothesis on the target $M$ is unnecessary. In particular, the first relation in Cor.~\ref{log_to}  holds for holomorphic maps into an arbitrary Hermitian manifold. 

\subsection{Proof of Theorem~\ref{qbounds}}
Suppose that $u$ is not anti-holomorphic. Then the differential $\partial u$ has only isolated zeroes $z_i\in\Sigma$; by $r'_i$ we denote their orders. Excising small disks $D_i$ centred at $z_i$, by Stokes' formula we obtain
$$
\int_{\Sigma\backslash\cup D_i}\Delta\log e'(u)\omega_\Sigma=
\sum_i\int_{\partial D_i}\frac{\partial}{\partial n}\log e'(u)ds.
$$
Parameterising the boundary $\partial D_i$ as $r\cdot e^{i\theta}$, where $\theta\in [0,2\pi]$, we compute the integrals on the right-hand side:
$$
\int_{\partial D_i}\frac{\partial}{\partial n}\log e'(u)ds=\int_0^{2\pi}\frac{\partial}{\partial r}\log
r^{2r'_i}rd\theta+O(r)=4\pi r'_i+O(r).
$$
Now integrating the first formula in Cor.~\ref{log_to} over $\Sigma\backslash\cup D_i$, and letting the radii of the $D_i$'s to zero, we obtain
$$
\int_\Sigma q'(u)\omega_\Sigma\geqslant\pi\sum_{z_i}r'_i+\frac{1}{2}\int_\Sigma K_\Sigma\omega_\Sigma.
$$
By the Gauss-Bonnet formula the last integral equals $4\pi(1-\varrho)$, and we finally obtain
$$
Q'_+(u)\geqslant\pi\left(\sum_{z_i}r'_i+2-2\varrho\right).
$$
In the case when a map $u$ is not holomorphic, the analogous bound for $Q''_+(u)$ is derived in a similar fashion.
\qed

\section{Proof of Theorem~\ref{MT}: curvature and energy concentration}
\label{proofMT0}
\subsection{Preliminaries: key lemma}
Let $D$ be a unit disk in $\mathbf C$, equipped with a Euclidean metric, and $dV$ be its Lebesgue measure. By $D_{1/2}$ we denote a disk centred at the origin of radius $1/2$. The following bound for the $L_p$-norm of a positive function in terms of its $L_1$-norm is an important ingredient in the proof of Theorem~\ref{MT}. It is essentially a consequence of Reshetnyak's analysis in~\cite{Re} combined with a version of Schwarz lemma for subharmonic functions. We state it in the form close to the one in~\cite{Tro}.
\begin{lemma}
\label{ML} 
Let $\varphi\in L^1(D)$ be a function such that $\Delta\varphi\in L^1(D)$. Suppose that
$$
\int_D(\Delta\varphi)^+dV\leqslant\kappa<4\pi,
$$
where $(\Delta\varphi)^+=\max\{\Delta\varphi,0\}$. Then for any $1\leqslant p<4\pi/\kappa$ there exists a constant $C(p,\kappa)$ such that
$$
\abs{e^\varphi}_{L^p(D_{1/2})}\leqslant C(p,\kappa)\cdot\abs{e^\varphi}_{L^1(D)}.
$$
\end{lemma}

\noindent
Below we outline the proof of the lemma. First, recall that for any compactly supported measure $\mu$ on $\mathbf C$ a function $v$ is called its {\em logarithmic potential}, if it solves 
$\Delta v=\mu$ among locally integrable functions and
$$
(2\pi v(z)+\mu(\mathbf C)\log\abs{z})\to 0\quad\text{as}\quad\abs{z}\to\infty.
$$
Such a potential is unique, and is given by the Riesz formula
\begin{equation}
\label{Riesz}
v(z)=-\frac{1}{2\pi}\int_{\mathbf C}\log\abs{z-\zeta}d\mu(\zeta).
\end{equation}
The following $L_p$-estimate on $e^v$ is a version of a statement in~\cite[p.~251]{Re} (Theorem~3.1). The $L_p$-norm below is understood in the sense of the Lebesgue measure.
\begin{prop}
\label{p1}
Let $\mu$ be a compactly supported measure and $v$ be its logarithmic potential. Then for any $p\geqslant 1$ such that $0<\mu(D)<4\pi/p$ we have
$$
\abs{e^v}_{L^p(D)}\leqslant\left(\frac{2\pi}{\delta+2}2^{\delta+2}\right)^{1/p}
$$
where $\delta=(-p\mu(D)/2\pi)$.
\end{prop}
\begin{proof}
By Jensen's inequality~\cite{LL}, we have
$$
\exp\left(\frac{1}{\mu(D)}\int_D\log\abs{z-\zeta}^\delta d\mu(\zeta)\right)\leqslant
\frac{1}{\mu(D)}\int_D\abs{z-\zeta}^\delta d\mu(\zeta).
$$
By the Riesz formula, the left-hand side here is precisely $e^{pv}$, and Fubini's theorem yields
$$
\abs{e^v}^p_{L^p(D)}\leqslant\frac{1}{\mu(D)}\int_D\int_D\abs{z-\zeta}^\delta dV(z)d\mu(\zeta),
$$
where the interior integral is bounded by $(2\pi\cdot 2^{\delta+2})/(\delta+2)$.
\end{proof}

For a proof of Lemma~\ref{ML} we also need the following version of Schwarz lemma; its proof can be found in~\cite{Tro}.
\begin{prop}
\label{p2}
Let $w$ be a subharmonic function on the unit disk $D$, $\Delta w\leqslant 0$. Then for any $z\in D$ we have
$$
e^{w(z)}\leqslant\frac{1}{\pi(1-\abs{z}^2)^2}\int_D e^wdV.
$$
\end{prop}
\begin{proof}[Proof of Lemma~\ref{ML}.]
Define a measure $\mu$ on a unit disk by setting $d\mu$ to be equal to $(\Delta\varphi)^+dV$.
Denote by $v$ its logarithmic potential and by $w$ the difference $(\varphi-v)$. Clearly, the latter is subharmonic:
$$
\Delta w=-(\Delta\varphi)^-\leqslant 0,
$$
where $(\Delta\varphi)^-=\max\{-\Delta\varphi,0\}$. By Riesz formula~\eqref{Riesz}, we have
$$
v(z)\geqslant-\frac{\mu(D)}{2\pi}\log(1+\abs{z})\geqslant-\frac{\kappa}{2\pi}\log 2,
$$
and hence
$$
\int_De^wdV\leqslant 2^{\kappa/2\pi}\int_De^\varphi dV.
$$
Applying Prop.~\ref{p2} to $w$, we conclude that
$$
e^{w(z)}\leqslant C_1(\kappa)\int_De^\varphi dV
$$
for any $z\in D_{1/2}$. On the other hand, by Prop.~\ref{p1}, we have
$$
\abs{e^v}_{L^p(D)}\leqslant C_2(p,\kappa).
$$
Now the claim follows by the combination of the last two inequalities:
$$
\abs{e^\varphi}_{L^p(D_{1/2})}\leqslant\sup_{D_{1/2}}e^w\cdot \abs{e^v}_{L^p(D)}\leqslant
C(p,\kappa)\abs{e^\varphi}_{L^1(D)}.
$$
\end{proof}
\subsection{Controlling the energy via $Q_+(u)$}
We proceed with the proof of Theorem~\ref{MT}. Here we show that the discrete part of the limit measure $dQ_+$ is supported precisely at the set of bubble points $\{x_1, \ldots, x_k\}$. Let $x$ be a point that belongs to the support of the discrete part of $dQ_+$, that is $dQ_+(x)>0$. Then by inequalities~\eqref{CS1}-\eqref{CS2}, it is a bubble point. Now we prove the converse -- for every bubble point $x_i$ the mass $dQ_+(x_i)$ is positive and, moreover, is at least $\pi/2$.

Suppose the contrary: there exists a bubble point $x_i$ such that $dQ_+(x_i)<\pi/2$. Since the point $x_i$ is isolated in the discrete part of $dQ_+$, there exists an open disk $D\subset\Sigma$, centred at $x_i$, such that
$$
\int_Dq_+(u_n)d\mathit{Vol}_\Sigma+\frac{1}{2}\int_D\abs{K_\Sigma}d\mathit{Vol}_\Sigma
\leqslant\frac{\kappa}{4}<\frac{\pi}{2}
$$
for a sufficiently large $n$. Combining this with the relations in Cor.~\ref{log_to}, we conclude that
$$
\int_D\left(\Delta\log e'(u_n)\right)^+\mathit{dVol}_\Sigma\leqslant\kappa<2\pi,
$$
$$
\int_D\left(\Delta\log e''(u_n)\right)^+\mathit{dVol}_\Sigma\leqslant\kappa<2\pi.
$$
Assuming that the metric on $D$ is conformally Euclidean, we now apply Lemma~\ref{ML}, to conclude that
\begin{equation}
\label{Lp}
\abs{e(u_n)}_{L^p(D_{1/2})}\leqslant\abs{e'(u_n)}_{L^p(D_{1/2})}+\abs{e''(u_n)}_{L^p(D_{1/2})}
\leqslant 2C(p,\kappa)E(u_n),
\end{equation}
where $1\leqslant p<4\pi/\kappa$. In other words, since the energies are bounded, so are the $L^p(D_{1/2})$-norms of the $e(u_n)$'s.

For a proof of the claim it is sufficient to show that the sequence $e(u_n)$ is bounded in $W^{2, r}(D_{1/8})$ for some $r>1$. Indeed, embedding the latter space into $C^0(\bar D_{1/8})$, see~\cite{GT}, we conclude that the $e(u_n)$'s are uniformly bounded in $D_{1/8}$ what contradicts the supposition that $x_i$ is a bubble point.

We prove the boundedness in $W^{2,r}(D_{1/8})$ separately for the sequences $e'(u_n)$ and $e''(u_n)$.
First, we show that
\begin{equation}
\label{aux}
\abs{\Delta e'(u)}_{L^r(D_{1/4})}\leqslant C_1\abs{e(u_n)}_{L^{2r}(D_{1/2})}+C_2
\end{equation}
for some constants $C_1$ and $C_2$. By the first relation in Prop.~\ref{to}, for  this it is sufficient to get the corresponding estimates on the quantities $\abs{\beta'(u_n)}^2$ and $q'(u_n)e'(u_n)$. 
The estimate for the $L_r$-norm of the second quantity is a straightforward consequence of inequalities~\eqref{CS1}-\eqref{CS2}. The estimate for $\abs{\beta'(u_n)}^2$ follows from the form of the harmonic map equation together with Schauder estimates in Sobolev spaces. In more detail, embedding $M$ isometrically into a Euclidean space and using the extrinsic form of the harmonic map equation~\cite{Sch}, we conclude that
$$
\abs{\Delta u^i_n}_{L^r(D_{1/2})}\leqslant C\cdot\abs{e(u_n)}_{L^r(D_{1/2})},
$$
where $(u_n^i)$ are the coordinates of the map $u_n$ viewed as a map into the Euclidean space. Further, by Schauder estimates~\cite[Th.~9.11]{GT} we obtain
$$
\abs{u^i_n}_{W^{2,r}(D_{1/4})}\leqslant C_3\abs{\Delta u^i_n}_{L^r(D_{1/2})}+C_4,
$$
where $C_4$ depends on the embedding of $M$. The combination of these inequalities yields the $L^r$-bound on the second derivatives of the $u_n$'s, and hence the $L^{2r}$-bound on 
$\abs{\beta'(u_n)}^2$.

Now we set $p$ in inequality~\eqref{Lp} to be $2r$. Since $\kappa<2\pi$, choosing $p$ sufficiently close to $4\pi/\kappa$, we can suppose that $r>1$. Under the hypotheses of the theorem the energies $E(u_n)$ are bounded, and the combination of inequalities~\eqref{Lp} and~\eqref{aux} implies that so is the sequence $\Delta e'(u_n)$ in the space $L^r(D_{1/4})$.  Using Schauder estimates ~\cite[Th.~9.11]{GT} again, we obtain
$$
\abs{e'(u_n)}_{W^{2,r}(D_{1/8})}\leqslant C_5\abs{\Delta e'(u_n)}_{L^r(D_{1/4})}+
C_6\abs{e'(u_n)}_{L^r(D_{1/4})} .
$$
Thus, the sequence $e'(u_n)$ is indeed bounded in $W^{2,r}(B_{1/8})$ as was claimed. The boundedness of the sequence $e''(u_n)$ can be demonstrated in a similar fashion.
\qed

\section{Proof of Theorem~\ref{MT}: bubble tree decomposition}
\label{proofMT}
\subsection{Background material on the bubble tree construction}
We start with recalling the bubble tree construction for sequences of harmonic maps. It is used in sequel for the proof of the curvature identity for the limit measure $dQ_+$ in Theorem~\ref{MT}. We follow closely the exposition in the paper~\cite{Pa96} by Parker, where further details and references on the subject can be found.

The construction is based on the following, by now standard facts, proved in~\cite{SaU}.
\begin{prop}
\label{auxSaU}
For a Riemannian manifold $(M,h)$ there exist positive constants $C_*$ and $\varepsilon_*$
such that
\begin{itemize}
\item[(i)] If $u:D_{2r}\to M$ is a harmonic map from a Euclidean disk of radius $2r$ and its energy $E_{2r}(u)=\int_{D_{2r}}e(u)$ is not greater than $\varepsilon_*$, then 
$$
\sup_{D_r}\abs{du}^2\leqslant C_*r^{-2}E_{2r}(u).
$$
\item[(ii)] If $u_n:D_{2r}\to M$ is a sequence of harmonic maps with $E_{2r}(u_n)<\varepsilon_*$ for all $n$, then it contains a subsequence that converges in $C^1$.
\item[(iii)] Any non-trivial harmonic map $u:S^2\to M$ has energy $E(u)\geqslant\varepsilon_*$.
\item[(iv)] Any smooth finite-energy harmonic map from the punctured disk $D\backslash\{0\}$ to $M$ extends to a smooth harmonic map on $D$.
\end{itemize}
\end{prop}

The combination of these facts with a covering argument is sufficient to conclude that a sequence of harmonic maps $u_n$ of bounded energy converges in $C^1$ away from the bubble points $\{x_1,\ldots, x_k\}$, and 
\begin{equation}
\label{eq01}
e(u_n)d\mathit{Vol}_\Sigma\rightharpoonup e(u)d\mathit{Vol}_\Sigma+\sum_{i=1}^km_i\delta_{x_i},
\end{equation}
see~\cite{SaU,Pa96}. Now we focus on one bubble point $x_i$. The bubble tree construction involves a few steps and depends on a ``renormalisation constant'' $C_R<\varepsilon_*/2$, which  will be chosen to be sufficiently small, see Prop.~\ref{2props}. Below we adopt the convention of denoting a subsequence of $(u_n)$ by the same symbol. The proofs of auxiliary Claims~\ref{c1}--\ref{c4} we refer to~\cite[Sect.~6]{Pa96}.

\medskip
\noindent
{\em Step~1: renormalisation.} Let $D$ be a neighbourhood of $x_i$ that does not contain other bubble points. We can suppose that the metric on $D$ is conformally Euclidean and regard $D$ as a disk $D(0,4\rho)$ in the Euclidean plane. By $dV$ we denote the Lebesgue measure on $D$. Let $\varepsilon_n$ be the largest number such that $\varepsilon_n\leqslant\min\{\rho,1/n\}$ and
\begin{equation}
\label{eq02}
\int_{D(0,2\varepsilon_n)}e(u)dV\leqslant\frac{m_i}{16n^2}.
\end{equation}
Denote by $D_n$ the disks $D(0,2\varepsilon_n)$ and by $c_n=(c_n^1,c_n^2)$ the centre of mass corresponding to the measure $e(u_n)dV$ on $D_n$,
\begin{equation}
\label{eq03}
c_n^j=\left(\int_{D_n}x^je(u_n)dV\right)/\left(\int_{D_n}e(u_n)dV\right).
\end{equation}
Further, let $\lambda_n$ be the smallest $\lambda$ such that
$$
\int_{D(c_n,\varepsilon_n)-D(c_n,\lambda)}e(u_n)dV\geqslant C_R,
$$
where $C_R$ is the renormalisation constant.
\begin{claim}
\label{c1}
After passing to a subsequence, we have $\abs{c_n}\leqslant\varepsilon_n/2n^2$, $\lambda_n\leqslant\varepsilon_n/n^2$. Besides, there exists a constant $C$ such that the image $u_n(\partial D_n)$ lies in the ball $B(u(x),C/n)$ in $M$.
\end{claim}
We use the centres of masses $c_n$'s and the scales $\lambda_n$'s to define the renormalisations
$$
R_n(z)=\lambda_n\cdot z+c_n,
$$
and the renormalised maps
$$
\tilde u_n=R_n^*u_n:S_n\to M,
$$
where $S_n=R_n^{-1}D_n$. By Claim~\ref{c1} and the definitions of the $c_n$'s and $\lambda_n$'s, the domains $S_n$ exhaust the whole Euclidean plane. By the conformal invariance, the energies $E(\tilde u_n)$ are also bounded. Further, we have the following statement.
\begin{claim}
\label{c2}
The renormalised maps $\tilde u_n$ are harmonic with respect to the Euclidean metric on $\mathbf R^2$. The measures $e(\tilde u_n)dV$ have centres of mass in the origin and satisfy the relations
$$
\lim_n\int_{S_n}e(\tilde u_n)dV=m_i\qquad\text{and}
\qquad\lim\int_{\mathbf R^2\backslash D(0,1)}e(\tilde u_n)dV=C_R.
$$
\end{claim}

\medskip
\noindent
{\em Step~2: constructing a bubble.}
Now we find a convergent subsequence of $(\tilde u_n)$. Choose a sequence $K_m$ of compact sets that exhaust $\mathbf R^2$. For each $m=1,2,\ldots$ one can apply Prop.~\ref{auxSaU} together with a covering argument to find a subsequence $(\tilde u_n)$ that converges away from a finite number of points. Taking a diagonal subsequence we obtain a new sequence that converges in $C^1$ on compact sets in $\mathbf R^2\backslash\{y_1,\ldots, y_\ell\}$ to a smooth map $\tilde u:\mathbf R^2\to M$.

Identify $\mathbf R^2$ via the stereographic projection with the unit sphere $S^2\subset\mathbf R^3$; we suppose that the north pole $p^+$ is mapped to the origin, the south pole $p^-$ to infinity, and the equator to the unit circle. The conformal invariance of energy implies that the limit map $\tilde u$ has finite energy, and by the statement~$(iv)$ in Prop.~\ref{auxSaU} extends to a harmonic map from $S^2$ to $M$, called {\em bubble sphere}. Also by the convergence of the $\tilde u_n$'s, we conclude
\begin{equation}
\label{eq1}
e(\tilde u_n)d\mathit{Vol}_{S^2}\rightharpoonup e(\tilde u)d\mathit{Vol}_{S^2}+\sum_{j=1}^\ell\bar m_j\delta_{y_j}+\nu(x_i)\delta_{p^-},
\end{equation}
where each $\bar m_j\geqslant\varepsilon_*$.
\begin{claim}
\label{c3} 
Each secondary bubble point $y_i$ lies in the northen hemisphere of $S^2$. Further,
if $E(\tilde u)<\varepsilon_*$, then $\tilde u$ is a constant map and either there are at least two secondary bubble points, $\ell\geqslant 2$, or there is only one bubble point and $\nu(x_i)=C_R$.
\end{claim}
We can now iterate this renormalisation procedure, thus obtaining bubbles on bubbles. By Claim~\ref{c3} each iteration decreases the corresponding concentration masses by at least $C_R$. Hence, the process terminates after a finite number of iterations. The result is a finite tree of bubbles whose vertices are harmonic maps (the limit map $u$ and the bubble spheres) and the edges are bubble points. We now describe a partitioning of the initial maps $u_n$ that keeps a more careful track of what is happening near the south pole.

\medskip
\noindent
{\em Step~3: partitioning.}
Inside the disk $D_n$ we identify three domains: the disk $D(c_n,\varepsilon_n)$, the smaller disk $D(c_n,n\lambda_n)$, and the annulus $A_n$ between them. Considering each bubble point $x_i$ and restricting the initial maps on these domains yields the following sequences of maps:

\smallskip
\noindent
{\em Base maps.} These are obtained by restricting the maps $u_n$ to the complement of $D(c_n,\varepsilon_n)$ for each bubble point. The image of the boundary near the point $x_i$, by Claim~\ref{c1} is mapped into the ball $B(u(x_i),C/n)$. We extend $u_n$ over $D(c_n,\varepsilon_n)$ by coning off the image:
\begin{equation}
\label{eq04}
\bar u_n(r,\theta)=\frac{r}{\varepsilon_n}u_n(\varepsilon_n,\theta),
\end{equation}
where $(r,\theta)$ are polar coordinates in $D(c_n,\varepsilon_n)$, centred at $c_n$, and the multiplication on the right-hand side is understood in geodesic coordinates in $B(u(x_i),C/n)$.
Performing this at each bubble point, we obtain maps $\bar u_n:\Sigma\to M$.

\smallskip
\noindent
{\em Bubble maps.} Restricting $u_n$ to the disk $D(c_n,n\lambda_n)$ and renormalising gives the map $\tilde u_n$ from $D(0,n)$ to $M$. Regarding the latter as a map from the unit sphere and coning over the south pole, yields a bubble map $\overline{Ru_{n,i}}:S^2\to M$, where the index $i$ refers to the bubble point $x_i$ under consideration.

\smallskip
\noindent
{\em Neck maps.} These are restrictions of the initial maps $u_n$'s to the annular region $A_n$, that is the difference $D(c_n,\varepsilon_n)\backslash D(c_n,n\lambda_n)$.

The base maps are constructed to converge to the limit map $u$. The maps $\tilde u_n$ of Step~1 are now decomposed into the bubble maps, which converge on $S^2$ to the bubble $\tilde u$, and the neck maps, which are pushed to the south pole as $n\to+\infty$ and account for the term $\nu(x_i)\delta_{p^-}$ in~\eqref{eq1}. More precisely, we have the following statement.
\begin{claim}
\label{c4}
After a selection of a subsequence, we have:
\begin{itemize}
\item[(i)] $\bar u_n\to u$ in $W^{1,2}\cap C^0$ on the surface $\Sigma$;
\item[(ii)] $\overline{Ru_{n,i}}\to\tilde u$ in $W^{1,2}\cap C^0$ on compact sets in the complement of $\{y_1,\ldots,y_\ell\}$ in $S^2$;
\item[(iii)] the energy concentration $\nu(x_i)$ at the south pole is equal to $\lim\sup E(\left.u_n\right|A_n)$;
\item[(iv)] the following energy identity holds
$$
E(u_n)\to E(u)+\sum_{i=1}^k\left(\nu(x_i)+\lim_nE(\overline{Ru_{n,i}})\right).
$$
\end{itemize}
\end{claim}

We end the discussion with the following statement that describes two fundamental properties of the bubble convergence: no energy loss in the necks and zero distance bubbling. For a proof we refer to~\cite{Pa96}.
\begin{prop}
\label{2props}
For a sufficiently small renormalisation constant $C_R$ at each bubble point $x_i$ there is no energy loss in the necks-- the mass $\nu(x_i)$ vanishes, and hence we have
$$
E(u_n)\to E(u)+\sum_{i=1}^k\lim_nE(\overline{Ru_{n,i}}).
$$
Further, the neck length $\diam(u_n(A_n))$ converges to zero. In particular, at each bubble point $x_i$ the images of the base map $u$ and the bubble sphere $\tilde u$ meet at $u(x_i)=\tilde u(p^-)$.
\end{prop}

\subsection{Curvature concentration respects the bubble tree}
Now we finish the proof of Theorem~\ref{MT}. First, since the maps $u_n$ converge to the limit map $u$ in $C^1$-topology on compact subsets of the complement $\Sigma\backslash\{x_1,\ldots, x_k\}$ one concludes that
$$
q_+(u_n)d\mathit{Vol}_\Sigma\rightharpoonup q_+(u)\mathit{Vol}_\Sigma
+\sum_{i=1}^kq_i\delta_{x_i}.
$$
Besides, as was proved in Sect.~\ref{proofMT0}, each $q_i\geqslant\pi/2$. Below we go through the steps of the bubble tree construction to show that each $q_i$ is a sum of $Q_+(\phi)$ over all bubble spheres $\phi$ in the bubble tree corresponding to the point $x_i$. The estimate $Q_+(\phi)\geqslant 2\pi$, and hence the improvement $q_i\geqslant 2\pi$, follow then from Cor.~\ref{cor:qbounds}.

\medskip
\noindent
{\em Step~1: renormalisation.} Let $x_i$ be a bubble point and $\tilde u_n=R_n^*u_n$ be the renormalised maps. We claim that the following relation holds:
$$
\lim_n\int_{S_n}q_+(\tilde u_n)dV=q_i.
$$
Indeed, it is straightforward to check that
\begin{equation}
\label{peq1}
\int_{S_n}q_+(\tilde u_n)dV=\int_{D_n}q_+( u_n)dV.
\end{equation}
Up to a selection of a subsequence, we can suppose that the original maps $u_n$ are such that
\begin{equation}
\label{peq2}
\frac{8n^2-1}{8n^2}q_i\leqslant\int_{B_n}q_+(u_n)dV\leqslant\frac{8n^2+1}{8n^2}q_i
\quad\text{and}\quad\int_{D_n\backslash B_n}q_+(u_n)dV\leqslant\frac{q_i}{8n^2}.
\end{equation}
This follows by repeating the argument in~\cite[p.~628]{Pa96}, used to prove Claim~\ref{c1}. These relations imply that the right-hand side in~\eqref{peq1} converges to $q_i$, thus demonstrating the claim.

\medskip
\noindent
{\em Step~2: constructing a bubble.}
In this step we find a subsequence of $(\tilde u_n)$ that converges on compact subsets of $\mathbf R^2\backslash\{y_1,\ldots, y_\ell\}$ to a smooth map $\tilde u:\mathbf R^2\to M$. The latter, regarded as a map from $S^2$, is a bubble sphere. Viewing all maps $\tilde u_n$ as maps from $S^2$, we have
$$
e(\tilde u_n)d\mathit{Vol}_{S^2}\rightharpoonup e(\tilde u)d\mathit{Vol}_{S^2}
+\sum_{j=1}^\ell\bar m_j\delta_{y_j},
$$
where each $\bar m_j\geqslant\varepsilon_*$. Here we used Prop.~\ref{2props}, according to which there is no energy loss at the south pole. Repeating the arguments in Sect.~\ref{proofMT0}, we obtain
$$
q_+(\tilde u_n)d\mathit{Vol}_{S^2}\rightharpoonup q_+(\tilde u)d\mathit{Vol}_{S^2}
+\sum_{j=1}^\ell\bar q_j\delta_{y_j}+\eta(x_i)\delta_{p^-},
$$
where each $\bar q_i\geqslant\pi/2$, and the term $\eta(x_i)$ accounts for a possible curvature loss at the south pole.

\medskip
\noindent
{\em Step~3: partitioning.}
Now we prove the following two statements:
{\em
\begin{itemize}
\item[(i)] the curvature concentration $\eta(x_i)$ at the south pole is equal to $\lim\sup Q_+(\left.u_n\right|A_n)$;
\item[(ii)] the following curvature identity holds:
$$
Q_+(u_n)\to Q_+(u)+\sum_{i=1}^k\left(\eta(x_i)+\lim_nQ_+(\overline{Ru_{n,i}})\right).
$$
\end{itemize}
}
First, we note that the amount of curvature on the cone extensions, used to define base and bubble maps, is small. 
More precisely, the energy of the cone extension~\eqref{eq04} can be estimated as
$$
\int_{D(c_n\varepsilon_n)}e(\bar u_n)dV\leqslant C/n^2
$$
for some constant $C$, see~\cite[p.~630]{Pa96}; a similar estimate also holds for the cone extension of the bubble maps $\overline{R}u_n$. Now combining inequalities~\eqref{CS1}-\eqref{CS2} with the estimate above, we obtain
$$
\int_{D(c_n,\varepsilon_n)}q_+(\bar u_n)dV\leqslant C/n^2
$$
for some constant $C$; similarly, one can estimate the curvature quantity for the cone extension of $\overline Ru_{n,i}$. Since the partitioning amounts for all the curvature, we conclude that
$$
Q_+(u_n)\to Q_+(u)+\lim_n\sum_{i=1}^k\left(Q_+(\left.u_n\right|A_{n,i})+Q_+(\overline{Ru_{n,i}})\right).
$$
The statement~$(i)$ follows by repeating the argument in~\cite[p.~631]{Pa96}, used to prove Claim~\ref{c4}, in combination with inequalities~\eqref{CS1}-\eqref{CS2}. The relation above together with the statement~$(i)$ yield the statement~$(ii)$.

We end with explaining that at each bubble point $x_i$ there is no curvature loss in the neck -- the quantity $\eta(x_i)$ vanishes. Indeed, using inequalities~\eqref{CS1}-\eqref{CS2} again, we have
$$
Q_+(\left.u_n\right|A_{n,i})\leqslant 2\sqrt{2}\max\abs{\Omega}\cdot E(\left.u_n\right|A_{n,i}).
$$
Thus, the statement~$(i)$ and Claim~\ref{c4}, yield
$$
\eta(x_i)\leqslant 2\sqrt{2}\max\abs{\Omega}\cdot\nu(x_i), 
$$
where $\nu(x_i)$ is a possible energy loss in the neck. By Prop.~\ref{2props} the quantity $\nu(x_i) $ vanishes, and hence so does $\eta(x_i)$. This shows that the curvature mass $q_i$ satisfies the relation 
$$
q_i=Q_+(\tilde u)+\sum_{j=1}^\ell\bar q_j,
$$ 
where $\tilde u$ is a bubble sphere and the sum is taken over the secondary bubble points. Iterating the procedure, we obtain that $q_i$ is indeed a sum of $Q_+(\phi)$ over all bubble spheres $\phi$ in the bubble tree at $x_i$.
\qed

\section{A few remarks}
\medskip
\noindent
{\em 1.} Mention that when the sequence $u_n$ in Theorem~\ref{MT} consists of holomorphic or anti-holomorphic maps, then the K\"ahler hypothesis on $M$ is unnecessary.
More precisely the following statement holds:

\medskip
\noindent
{\em Let $\Sigma$ and $M$ be a closed Riemannian surface and a closed Hermitian manifold respectively. Let $u_n:\Sigma\to M$ be a sequence of holomorphic maps of bounded energy that converges to a holomorphic map $u$ in the sense of bubble convergence. Then the limit curvature measure $dQ_+$ satisfies the conclusions of Theorem~\ref{MT}.}

\medskip
\noindent
The proof follows the same line of argument as in Sect.~\ref{proofMT0} and~\ref{proofMT}. The key point is that the Bochner identities in Sect.~\ref{prem} continue to hold for holomorphic maps into an arbitrary Hermitian manifold. So does the energy estimate
$$
E(u)\geqslant 4\pi/H,
$$ 
where $H$ is an upper bound for the holomorphic sectional curvature,  for holomorphic spheres.

\medskip
\noindent
{\em 2.} If a metric on the domain surface $\Sigma$ is allowed to vary, and the corresponding complex structures stay in a compact region in the moduli space, then the Sacks-Uhlenbeck bubble convergence theorem still continues to hold, see~\cite{Pa96,CT}, and so does Theorem~\ref{MT}. However, if the complex structures leave each bounded region, then the bubble tree convergence can fail; in particular, there can be energy loss in the limit. It is interesting to know whether this energy loss can be detected by the curvature quantity $Q_+(u)$.

\medskip
\noindent
{\em 3.} The no curvature loss in the necks statement, explained in Sect.~\ref{proofMT},  is essentially a consequence of the corresponding statement for the energy, see Prop.~\ref{2props}. However, it is interesting to know its proof  that does not rely on Prop.~\ref{2props}.

\end{document}